\documentclass[a4paper,10pt]{article}
\usepackage{amssymb}
\usepackage{amsmath}
\usepackage{amsfonts}
\usepackage{amsthm}
\usepackage{latexsym}
\usepackage{graphicx}
\usepackage{hyperref}

\newtheorem{theorem}{Theorem}[section]
\newtheorem{lemma}[theorem]{Lemma}
\newtheorem{remark}[theorem]{Remark}
\newtheorem{proposition}[theorem]{Proposition}
\newtheorem{corollary}[theorem]{Corollary}

\begin{document}
\title{Klein-Gordon-Maxwell System\\ in a bounded domain
\footnote{The authors are supported by M.I.U.R. - P.R.I.N.
``Metodi variazionali e topologici nello studio di fenomeni non lineari''.}
}

\author{Pietro d'Avenia\\ Dipartimento di Matematica\\ Politecnico di Bari\\pdavenia@poliba.it
\and
Lorenzo Pisani, Gaetano Siciliano\\ 
Dipartimento di Matematica\\ Universit\`{a} degli Studi di Bari\\
pisani@dm.uniba.it, siciliano@dm.uniba.it
}
\maketitle

\begin{abstract}
This paper is concerned with the Klein-Gordon-Maxwell system in a bounded spatial domain. We
discuss the existence of standing waves $\psi=u(x)e^{-i\omega t}$ in
equilibrium with a purely electrostatic field $\mathbf{E}=-\nabla\phi(x)$. We
assume an homogeneous Dirichlet boundary condition on $u$ and an inhomogeneous
Neumann boundary condition on $\phi$. In the \textquotedblleft
linear\textquotedblright\ case we characterize the existence of nontrivial
solutions for small boundary data. With a suitable nonlinear perturbation in
the matter equation, we get the existence of infinitely many solutions.

\textbf{Mathematics Subject Classification 2000:} 35J50, 35J55, 35Q60
\end{abstract}

\section{Introduction}

Many recent papers show the application of global variational methods to the study of
the interaction between matter and electromagnetic fields.
A typical example is given by the Klein-Gordon-Maxwell (KGM for short) system.

We consider a matter field $\psi$, whose free
Lagrangian density is given by%
\begin{equation}                                                                    \label{l0}
 \mathcal{L}_{0}=\frac{1}{2}\left(  \left\vert \partial_t\psi \right\vert
^{2}-\left\vert \nabla\psi\right\vert ^{2} - m^{2}\left\vert \psi\right\vert
^{2}\right),
\end{equation}
with $m>0$. The field is charged and in equilibrium with its own electromagnetic field
$(\mathbf{E},\mathbf{B})$, represented by means of the gauge potentials $(\mathbf{A}%
,\phi)$,%
\begin{eqnarray*}
\mathbf{E} &  = & -\left(\nabla \phi+{\partial_t\mathbf{A}}\right),\\
\mathbf{B} &  = & \nabla\times\mathbf{A}.
\end{eqnarray*}
Abelian gauge theories provide a model for the interaction; formally we
replace the ordinary derivatives $\left(\partial_t,\nabla\right)$ in (\ref{l0}) with the so-called gauge covariant derivatives
\[
\left(\partial_t +iq\phi,\nabla -iq\mathbf{A}\right),
\]
where $q$ is a nonzero coupling constant (see \emph{e.g.}
\cite{fel}). Moreover, we add the
Lagrangian density associated with the electromagnetic field%
\[
\mathcal{L}_{1} = \frac{1}{8\pi}\left(  \left\vert \mathbf{E}\right\vert
^{2}-\left\vert \mathbf{B}\right\vert ^{2}\right)  .
\]
The KGM system is given by the Euler-Lagrange equations
corresponding to the total Lagrangian density
\[
\mathcal{L}=\mathcal{L}_{0}(\psi,\mathbf{A},\phi)+\mathcal{L}_{1}   (\mathbf{A},\phi).
\]

The study of the KGM system is carried out for special classes of
solutions (and for suitable classes of lower order nonlinear
perturbation in $\mathcal{L}_{0}$).
In this paper we consider%
\begin{align*}
\psi&  =u(x)e^{-i\omega t},\\
\phi&=\phi\left(x\right),\\
\mathbf{A}&=\mathbf{0},
\end{align*}
that is a standing wave in equilibrium with a purely electrostatic field
\begin{align*}
\mathbf{E} &  =-\nabla\phi\left(x\right),\\
\mathbf{B} &  =\mathbf{0}.
\end{align*}
Under this ansatz, the KGM system reduces to
\begin{equation}                                                                   \label{lin}
\left\{
\begin{array}{l}
-\Delta u-\left(  q\phi-\omega\right)  ^{2}u+m^{2}u=0,\\
\Delta\phi=4\pi q\left(  q\phi-\omega\right)  u^{2},
\end{array}
\right.
\end{equation}
(see \cite{bf2} or \cite{bf3} where the complete set of equations has been deducted).

We shall study (\ref{lin}) in a bounded domain
$\Omega\subset\mathbf{R}^{3}$ with smooth boundary $\partial\Omega$.
The unknowns are the real functions $u$ and $\phi$ defined on
$\Omega$ and the frequency $\omega\in\mathbf{R}$.
Throughout the paper we assume the following boundary conditions%
\begin{subequations}                                                                \label{113}
\begin{eqnarray}
 && u\left(x\right)  =0,
\\                                                                  \label{113neum}
 && \frac{\partial\phi}{\partial \mathbf{n}} \left(x\right)  =h\left(x\right).
\end{eqnarray}
\end{subequations}
The problem (\ref{lin}) has a variational structure and we apply
global variational methods.

First we notice that the system is symmetric with respect to $u$, that is, the pair
$(u,\phi)$ is a solution if and only if $(-u,\phi)$ is a solution.

Moreover, due to the Neumann condition (\ref{113neum}), the existence of solutions is independent
on the frequency $\omega$. Indeed the pair $(u,\phi)$ is a solution of (\ref{lin})-(\ref{113}) if and only if the pair $(u,\phi-\omega/q)$ is
a solution of the following problem
\begin{subequations}                                                 \label{problin}
 \begin{alignat} {2}                                                  \label{problin1}
 & -\Delta u-q^{2}\phi^{2}u+m^{2}u    =0 &  \textrm{in }\Omega,
\\
                                                                     \label{problin2}
 & \Delta\phi =4\pi q^{2}\phi u^{2} & \quad \textrm{in }\Omega,
 \end{alignat}
\end{subequations}
with the same boundary conditions (\ref{113}). In other words, for any
$\omega\in\mathbf{R}$, the existence of a standing wave $\psi=
u(x)e^{-i\omega t}$ in equilibrium with a purely electrostatic field
is equivalent to the existence of a static matter field $u(x)$, in
equilibrium with the same electric field. So we focus our attention on the problem
(\ref{problin}).

The boundary datum $h$ plays a key role.

If $h=0$, then it is easy to see that the system (\ref{problin})-(\ref{113}) have only the solutions $u=0,\phi = const$.

If $\int_{\partial\Omega}h~d\sigma=0$, then (\ref{problin})-(\ref{113})
has infinitely many solutions corresponding to $u=0$. Such solutions have the form $u=0,\phi=\chi+const$
(see Lemma \ref{defchi} below, where $\chi$ is introduced) and we call them \emph{trivial}.
In this case we are interested in finding nontrivial solutions (\emph{i.e.} solutions with $u\neq 0$).

On the other hand, it is well known that the Neumann condition gives rise to a necessary
condition for the existence of solutions of the boundary value problem. In our case, from (\ref{problin2})-(\ref{113neum}), we get%
\[
4\pi q^{2}\int_{\Omega}\phi u^{2}~dx=\int_{\partial\Omega}h~d\sigma.
\]
Hence, whenever $\int_{\partial\Omega}h~d\sigma\neq0$, solutions of (\ref{problin})-(\ref{113}), if any, are nontrivial.

The following theorem characterizes the existence of nontrivial solutions
for small boundary data.

\begin{theorem}                                                                     \label{thlin}
If $\left\|h\right\|_{H^{1/2}(\partial\Omega)}$ is sufficiently small (with respect to $m/q$), then the
problem (\ref{problin})-(\ref{113}) has nontrivial solutions $\left(
u,\phi\right)  \in H_{0}^{1}(\Omega)\times H^{1}(\Omega)$ if and only if
\[
\int_{\partial\Omega}h~d\sigma\neq0.
\]
\end{theorem}

We point out that the Lagrangian density ${\cal L}$ contains only the potential $W(|\psi|)= m^2 |\psi|^2/2$, which gives a positive energy (see the discussion about the energy in \cite{bf5}). Hence the solutions found in Theorem \ref{thlin} are relevant from the physical point of view.

Theorem \ref{thlin} shows that, if $q$ is sufficiently small, (\ref{problin})-(\ref{113}) has only the trivial solutions if and only if $\int_{\partial\Omega}h~d\sigma =0$. The same result holds true if $q=0$ (uncoupled system). It is immediately seen that, in the uncoupled case, if $\int_{\partial\Omega}h~d\sigma \ne 0$, then there exist no solutions at all.

Our second result is concerned with a nonlinear lower order perturbation in
(\ref{problin1}). So we study the following system
\begin{equation}                                    \label{probl}
\left\{
\begin{array}   {lll}
 -\Delta u-q^{2}\phi^{2}u+m^{2}u  =g(x,u) & & \textrm{in }\Omega,\\

 \Delta\phi  =4\pi q^{2}\phi u^{2} & & \textrm{in }\Omega,
\end{array}
\right.
\end{equation}
again with the boundary conditions (\ref{113}).
The nonlinear term $g$ is usually interpreted as a self-interaction among many
particles in the same field $\psi$.

We assume $g\in C\left(  \bar{\Omega}%
\times\mathbf{R},\mathbf{R}\right)  $ and

\begin{description}
\item[(g1)] $\exists\,a_{1},a_{2}\geq0,$ $\exists\,p\in\left(  2,6\right)  $
such that
\[
\left\vert g\left(  x,t\right)  \right\vert \leq a_{1}+a_{2}\left\vert
t\right\vert ^{p-1};
\]

\item[(g2)] $g\left(  x,t\right)  =o\left(  \left\vert t\right\vert \right)  $
as $t\rightarrow0$ uniformly in $x$;

\item[(g3)] $\exists\,s\in\left(  2,p\right]  $ and $r\geq0$ such that for
every $\left\vert t\right\vert \geq r$:
\[
0<sG\left(  x,t\right)  \leq tg\left(  x,t\right),
\]
where
\[
G\left(  x,t\right)  =\int_{0}^{t}g\left(  x,\tau\right)  \,d\tau.
\]

\end{description}

\begin{remark}
A typical nonlinearity $g$ satisfying $\mathbf{\left(  g_{1}\right)
}-\mathbf{\left(  g_{3}\right)  }$ is $g\left(  x,t\right)  =\left\vert
t\right\vert ^{p-2}t$, with $p\in\left(  2,6\right)  $.
\end{remark}

\begin{theorem}                                                                        \label{main}
Let $g$ satisfy $\mathbf{\left(  g_{1}\right)  }-\mathbf{\left(
g_{3}\right)  }$.

\begin{enumerate}
\item[a)] If $h\in H^{1/2}(\partial\Omega)$ is sufficiently small (with respect to
$m/q$) and satisfies
\begin{equation}
\label{116}
\int_{\partial\Omega}h~d\sigma=0,
\end{equation}
then the problem (\ref{probl}) has a nontrivial solution $\left(
u,\phi\right)  \in H_{0}^{1}(\Omega)\times H^{1}(\Omega)$.

\item[b)] If $g$ is odd, then, for every $h\in H^{1/2}(\partial\Omega)$
which satisfies (\ref{116}), problem (\ref{probl}) has infinitely many solutions $\left(
u_{i},\phi_{i}\right)  \in H_{0}^{1}(\Omega)\times H^{1}(\Omega)$,
$i\in\mathbf{N},$ such that
\[
\int_{\Omega}\left\vert \nabla u_{i}\right\vert ^{2}dx\rightarrow+\infty,
\]
whereas the set $\left\{  \phi_{i}\right\} $ is uniformly bounded in $H^{1}(\Omega)\cap L^\infty\left(\Omega\right)$.
\end{enumerate}
\end{theorem}

The present paper has been motivated by some results about the system
(\ref{probl}) in the case $\Omega = \mathbf{R}^{3}$. To the best of our
knowledge, our results are the first ones in the case of a bounded
domain. Under Dirichlet boundary conditions on both $u$ and $\phi$, the
existence results for (\ref{problin}) and (\ref{probl}) are analogous and simpler (see \cite{DPS2008}).

About the system (\ref{lin}) in $\mathbf{R}^{3}$, Theorem 1.1 in \cite{dm2} shows that there exists only the trivial solution.

In the case of a lower order nonlinear perturbation (problem (\ref{probl})), the pioneering result contained in  \cite{bf2} has been generalized in several papers: see \cite{bf3}, \cite{bf4}, \cite{dm1}.
Related results on analogous systems are contained in \cite{Cass}, \cite{D'Av-Pis}.

A different class of solutions for the KGM system is introduced in the papers \cite{bf3} and \cite{bf4}, where the authors show the existence of magnetostatic and electromagnetostatic solutions (3-dimensional vortices).

>From the physical point of view, the case of a positive lower order term
\[
    W(\left|\psi\right|)=\frac{1}{2} m^2 \left|\psi\right|^2 - G\left(x, \left|\psi\right|\right)
\]
is more relevant. This case is dealt with in some very recent papers (\cite{bf5}, \cite{bf6}, \cite{Long}).

Finally, we recall that global variational methods have been used also in the study of Schroedinger-Maxwell systems (see \emph{e.g.}  \cite{dm1}, \cite{ARuiz}, \cite{bf1}, \cite{Coc-Giorg}, \cite{D'Av}, \cite{PS1}).

\section{Functional setting}

The first step to study problems (\ref{problin}) and (\ref{probl}) is to reduce to
homogeneous boundary conditions. For the sake of simplicity, up to a simple rescaling, we can omit the constant $4\pi$.

\begin{lemma}\label{defchi}
For every $h\in H^{1/2}\left(\partial\Omega\right)$, let
\begin{equation*}
 \kappa=\frac{1}{\left|\Omega \right|}\int_{\partial\Omega}h\,d\sigma.
\end{equation*}
Then, there exists a unique $\chi\in H^2\left(\Omega\right)$ solution of
\begin{equation}                                                                             \label{chi}
\left\{
\begin{array}{ll}
\Delta\chi=\kappa & \textrm{in }\Omega,\\
\displaystyle{\frac{\partial\chi}{\partial\mathbf{n}}\left(  x\right)}
=h\left(  x\right) & \textrm{on }\partial\Omega, \vspace{3pt}\\
\displaystyle{\int_{\Omega}\chi\,dx=0}.
\end{array}
\right.
\end{equation}
\end{lemma}

\begin{remark}                                                                   \label{gil}
 It is well known that the solution of (\ref{chi}) satisfies
\begin{equation*}
 \left\|\chi \right\|_{H^2\left(\Omega\right)}\leq c\left(\left\|\kappa\right\|_2 +\left\|h\right\|_{H^{1/2}\left(\partial\Omega\right)}\right)
\end{equation*}
where $c$ is a positive constant.
So we obtain
\begin{equation*}                                                               \label{gilardi}
\left\|\chi\right\|_\infty\leq c_1\left\|h\right\|_{H^{1/2}\left(\partial\Omega\right)}.
\end{equation*}
\end{remark}

If we set
\begin{equation}                                                                 \label{cambio}
 \varphi=\phi-\chi,
\end{equation}
then (\ref{problin}) becomes
\begin{equation}
\left\{
\begin{array}{lll}
-\Delta u-q^2\left(  \varphi+\chi\right)  ^{2}u+m^{2}u=0 & &\textrm{in }\Omega,\\
\Delta\varphi=q^2\left(  \varphi+\chi\right)  u^{2}-\kappa & &\textrm{in }\Omega,\\
u\left(  x\right)  =0 &  & \textrm{on }\partial\Omega,\\
\displaystyle{\frac{\partial\varphi}{\partial\mathbf{n}}\left(  x\right)}
=0 & &\textrm{on }\partial\Omega.
\end{array}
\right.  \label{probchilin}
\end{equation}

Let us consider on $H^{1}_{0}\left(\Omega\right) $  the norm $\left\| \nabla u\right\| _{2} $ and on $H^{1}\left
( \Omega\right) $
\[
\left\| \varphi\right\| =\left( \left\| \nabla\varphi\right\| _{2}^{2}+ \left|
\bar\varphi\right| ^{2}\right) ^{1/2},
\]
where $\bar\varphi$ denotes the average of a function $\varphi$ on $\Omega$,
i.e.
\[
 \bar\varphi= \frac{1}{\left| \Omega\right| }\int_{\Omega}\varphi\, dx.
\]
Standard computations show that the solutions of (\ref{probchilin})
are  critical points of the $C^{1}$ functional
\begin{equation*}
  F\left(  u,\varphi\right) =  \frac{1}{2}\left\|\nabla
u\right\|^{2}_2+\frac{1}{2}\int_{\Omega}\left[  m^{2}-q^2\left(
\varphi+\chi\right)  ^{2}\right]  u^{2}dx-\frac{1}{2}\left\|
\nabla\varphi\right\|_2^{2}
+\kappa\left|\Omega\right|\bar\varphi,
\end{equation*}
defined in $H_{0}^{1}\left(\Omega\right)\times H^{1}\left(
\Omega\right)$.
Unfortunately it is strongly unbounded. We adapt a \emph{reduction argument} introduced in \cite{bf1}.
Let
\[
\Lambda=H_{0}^{1}\left( \Omega\right)\setminus\left\{0\right\}.
\]

\begin{lemma}                                                                              \label{lemmainvg}
For every $u\in \Lambda $ and $\rho\in L^{6/5}\left(\Omega\right)$ there exists a unique $\varphi \in H^{1}\left( \Omega\right) $ solution of
\begin{equation*}
\left\{
\begin{array}{ll}
-\Delta\varphi+q^2\varphi u^{2} =\rho & \textrm{in }\Omega,\\
\displaystyle{\frac{\partial\varphi}{\partial\mathbf{n}}\left(  x\right)  =0}
& \textrm{on } \partial\Omega.
\end{array}
\right.
\end{equation*}
\end{lemma}

\begin{proof}
Let $u\in \Lambda$ and $\rho\in L^{6/5}\left(\Omega\right)$ be fixed.
We shall apply the Lax-Milgram Lemma.

We consider the bilinear form
\[
a\left(\varphi,\zeta\right)=\int_{\Omega }\nabla\varphi\nabla\zeta\,dx
+q^2\int_{\Omega }\varphi\zeta u^2\,dx
\]
on $H^1\left(\Omega\right)$. By the H\"older and Sobolev inequalities, we get
\begin{eqnarray*}
a\left(\varphi,\zeta\right)
&\leq & \left\|\nabla\varphi\right\|_2\left\|\nabla\zeta\right\|_2
+q^2\left\|\varphi\right\|_3\left\|\zeta\right\|_3
\left\| u\right\|_6^2\\
&\leq & \left(1+c_1 \left\|u\right\|_6^2\right)\left\|\varphi\right\|\left\|\zeta\right\|
\end{eqnarray*}
and so $a$ is continuous.
Moreover,
\[
\lim_{\left\|\varphi\right\|\rightarrow+\infty}a\left(\varphi,\varphi\right)=+\infty.
\]
Indeed, if $\left\|\varphi\right\|\rightarrow+\infty$, we distinguish two cases.
\begin{enumerate}
\item If $\left\|\nabla\varphi\right\|_2\rightarrow+\infty$, then
\[
a\left(\varphi,\varphi\right)\geq\left\|\nabla\varphi\right\|_2^2\rightarrow+\infty.
\]
\item If $\left\|\nabla\varphi\right\|_2$
 is bounded, then
$\left|\bar\varphi\right|\rightarrow+\infty$. By the
Poincar\'e-Wirtinger inequality
\begin{equation*}
\left\| \varphi-\bar\varphi\right\| _{6} \leq c_2\left\|
\nabla\varphi\right\| _{2},
\end{equation*}
also $\left\|\varphi-\bar\varphi\right\|_2$ is bounded. Then we consider 
$\varphi= (\varphi-\bar\varphi)+\bar\varphi$ and obtain
\[
a\left(\varphi,\varphi\right)\geq q^2\left|\bar\varphi\right|^2\left\|u\right\|_2^2
-2q^2\left|\bar\varphi\right|\left\|\varphi-\bar\varphi\right\|_2\left\|u\right\|_4^2
\rightarrow +\infty.
\]
\end{enumerate}
By standard arguments, we deduce that the bilinear form $a$
is coercive in $H^1\left(\Omega\right)$.

On the other hand, by the Sobolev imbedding, we can consider the linear and continuous map
\[
\zeta\in H^1\left(\Omega\right) \longmapsto\int_{\Omega}\rho\zeta\,dx\in\mathbf{R}.
\]

The Lax-Milgram Lemma gives the assertion.
\end{proof}

So our reduction argument is based on the following result.

\begin{proposition}
For every $u\in \Lambda$
there exists a unique $\varphi_{u} \in H^{1}\left( \Omega\right) $ solution
of
\begin{equation}                                                                                  \label{120}
\left\{
\begin{array}{lll}
\Delta\varphi=q^2\left(  \varphi+\chi\right)  u^{2} -\kappa & & \textrm{in }\Omega,\\
\displaystyle{\frac{\partial\varphi}{\partial\mathbf{n}}\left(  x\right)  =0}
& &\textrm{on } \partial\Omega.
\end{array}
\right.
\end{equation}
Hence the set
\begin{equation}                                                                \label{graph}
 \{(u,\varphi)\in \Lambda\times H^1(\Omega) \mid
F'_\varphi\left(u,\varphi\right)=0 \}
\end{equation}
coincides with the graph of the map
$u\in \Lambda\mapsto\varphi_u\in H^{1}(\Omega)$.
\end{proposition}

\begin{proposition}
The map $u\in \Lambda
\mapsto\varphi_u\in H^{1}(\Omega)$ is $C^1$.
\end{proposition}
\begin{proof}
Since the graph of the map
$u\mapsto\varphi_u$ is given by (\ref{graph}),
we refer to the Implicit Function Theorem.

Straightforward calculations show that for every $\xi,\eta\in H^1(\Omega)$ and $w\in H^{1}_{0}(\Omega)$
\begin{eqnarray*}
F''_{\varphi\varphi}\left(u,\varphi\right)\left[\xi,\eta\right]&=&
-\int_\Omega \nabla\xi\nabla\eta\,dx-q^2\int_\Omega u^2\xi\eta\,dx, \\
F''_{\varphi u}\left(u,\varphi\right)\left[w,\eta\right]&=&-2q^2\int_\Omega
(\varphi+\chi)uw\eta\,dx.
\end{eqnarray*}
Then it is easy to see that $F''_{\varphi\varphi}$ and $F''_{\varphi u}$ are continuous.

On the other hand we have already seen that, for every
$(u,\varphi)\in \Lambda\times H^1(\Omega)$, the operator associated
to $F''_{\varphi\varphi}(u,\varphi)$ is invertible (Lemma
\ref{lemmainvg}). Hence the claim immediately follows.
\end{proof}

We can define on $\Lambda $ the \emph{reduced} functional
\begin{equation*}
J\left( u\right) =F\left( u,\varphi_{u}\right).
\end{equation*}
It is $C^1$ and it is easy to see that $(u,\varphi)\in\Lambda\times H^1(\Omega)$ is a critical point of $F$ if and only if $u$
is a critical point of $J$ and $\varphi=\varphi_u$.
So, to get nontrivial solutions of (\ref{problin}), we look for
 critical points of the functional $J$.

\bigskip
With the same change of variable (\ref{cambio}), problem
(\ref{probl}) becomes
\begin{equation}
\left\{
\begin{array}{lll}
-\Delta u-q^2\left(\varphi+\chi\right)^{2}u+m^{2}u-g\left(x,u\right)
=0 & &\textrm{in }\Omega,\\
\Delta\varphi=q^2\left( \varphi+\chi\right)  u^{2} -\kappa& &\textrm{in }\Omega,\\
u\left(  x\right)  =0 & &\textrm{on }\partial\Omega,\\
\displaystyle{\frac{\partial\varphi}{\partial\mathbf{n}}\left(  x\right)}
=0 & &\textrm{on }\partial\Omega.
\end{array}
\right.  \label{probchi}
\end{equation}
The solutions of (\ref{probchi}) are the critical points of the $C^1$-functional
\begin{equation*}
 F_{g}\left(  u,\varphi\right)=  F\left(  u,\varphi\right)-\int_{\Omega}G\left
 (x,u\right)dx
\end{equation*}
and, as above, we can consider the \emph{reduced} $C^1$-functional
\begin{equation}
J_{g}\left( u\right) =F_{g}\left( u,\varphi_{u}\right) \label{defj}.
\end{equation}
To get nontrivial solution of (\ref{probchi}) we look for critical points of $J_g$.

\section{Behavior of $\varphi_u$}

By Lemma \ref{lemmainvg}, for every $u\in \Lambda$, problem
\begin{equation}                                                                   \label{onlychi}
\left\{
\begin{array}{lll}
\Delta\xi-q^2\xi u^{2}=q^2\chi u^{2} & &\textrm{in }\Omega, \\
\displaystyle{\frac{\partial\xi}{\partial \mathbf{n}}=0}& &\textrm{on } \partial\Omega
\end{array}
\right.
\end{equation}
has a unique solution $\xi_u \in H^1\left(\Omega\right)$.

Analogously, for every $u\in \Lambda$, problem
\begin{equation}                                                                    \label{k}
\left\{
\begin{array}{lll}
\Delta\eta-q^2\eta u^{2}=-\kappa & &\textrm{in }\Omega,\\
\displaystyle{\frac{\partial\eta}{\partial \mathbf{n}}=0}& &\textrm{on } \partial\Omega.
\end{array}
\right.
\end{equation}
has a unique solution $\eta_u \in H^1\left(\Omega\right)$.

Of course, since the solution of (\ref{120}) is unique, we have
\begin{equation} \label{aggiunta1}
\varphi_{u}=\xi_{u}+\eta_{u}.
\end{equation}

\begin{lemma}[Properties of $\xi_{u}$]                                                   \label{zetaprop}
For every $u\in \Lambda$,
\begin{equation}                                                                            \label{intmeno}
 \int_\Omega \xi_u\chi u^2\,dx\leq 0
\end{equation}
and
\begin{equation}                                                                           \label{m}
-\max\chi\leq\xi_{u}\leq-\min\chi
\end{equation}
a.e. in $\Omega$.
\end{lemma}

\begin{proof}
Multiplying (\ref{onlychi}) by $\xi_u$ and integrating on $\Omega$,
we get immediately (\ref{intmeno}).

Moreover, if $\xi_u$ is the solution of (\ref{onlychi}), then $\xi_u+\min\chi$
is the unique solution of
\begin{equation*}
\left\{
\begin{array}{lll}
\Delta \xi =q^2 \left[ \xi +(\chi-\min\chi) \right]u^{2} & &\textrm{in }\Omega, \\
\displaystyle{\frac{\partial \xi}{\partial \mathbf{n}}\left( x\right) =0} & &\textrm{on }
\partial \Omega
\end{array}
\right.
\end{equation*}
and minimizes the functional
\[
f\left( \xi \right) =\frac{1}{2 }\int_{\Omega }\left| \nabla
\xi \right| ^{2}\,dx+\frac{q^2}{2}\int_{\Omega }\xi
^{2}u^{2}\,dx+q^2\int_{\Omega }\left(\chi-\min\chi\right) u^{2}\xi\,dx
\]
on $H^1\left(\Omega\right)$.
On the other hand
\[
f\left(-\left|\xi_{u}+\min\chi\right|\right)\leq f\left(\xi_{u}+\min\chi\right)
\]
and so
\[
\xi_{u}+\min\chi=-\left|\xi_{u}+\min\chi\right|,
\] 
a.e. in $\Omega$.
Hence $\xi_u\leq -\min\chi$, 
a.e. in $\Omega$.

Analogously, $\xi_u+\max\chi$
is the unique solution of
\[
\left\{
\begin{array}{lll}
\Delta \xi =q^2 \left[ \xi +(\chi-\max\chi) \right]u^{2} & &\textrm{in }\Omega,\\
\displaystyle{\frac{\partial \xi}{\partial \mathbf{n}}\left( x\right) =0} & &\textrm{on }\partial \Omega
\end{array}
\right.
\]
and, arguing as before, we get $\xi_u\geq -\max\chi$
a.e. in $\Omega$.\end{proof}

\begin{corollary}                                                                 \label{estimates}
For every $u\in \Lambda$,
\begin{equation}                                                                  \label{stimaphi}
 \left\| \xi_{u}\right\| _{\infty}\leq\left\| \chi\right\|_{\infty},
\end{equation}
\begin{equation}                                                                   \label{stima gradfi}
\left\| \nabla\xi_{u}\right\| _{2}\leq\left\| \nabla\chi\right\|_{2}.
\end{equation}
\end{corollary}

\begin{proof}
The inequality (\ref{stimaphi}) easily follows from
(\ref{m}). By (\ref{onlychi}),
$\xi_u$ satisfies
\[
\int_\Omega{\nabla\xi_u\nabla w\,dx}
+q^2\int_\Omega\left(\xi_u +\chi\right)u^2w\,dx=0
\]
for any $w\in H^1\left(\Omega\right)$. For $w=\xi_u+\chi$ we get
\[
\left\|\nabla\xi_u\right\|_2^2+\int_\Omega\nabla\xi_u\nabla\chi \,dx
+q^2\int_\Omega\left(\xi_u+\chi\right)^2u^2 \,dx=0
\]
from which one deduces (\ref{stima gradfi}).
\end{proof}

\begin{remark}                                             \label{unifbound}
We point out that, if $\kappa=0$, then $\varphi_u=\xi_u.$
Therefore (\ref{stimaphi}) and (\ref{stima gradfi}) become uniform estimates on $\varphi_u\in H^1\left(\Omega\right)\cap L^\infty\left(\Omega\right)$ and give rise to estimates on the \emph{old} variable
\[
\phi=\varphi_u+\chi=\xi_u+\chi.
\]
In other words, if $\int_{\partial\Omega}h~d\sigma = 0$, the solutions $\phi$ of (\ref{problin2})-(\ref{113neum}) are uniformly bounded with respect to $u\ne 0$.
>From (\ref{m}) we deduce also a more precise estimate
\begin{equation*}
\left\|\phi\right\|_\infty = \left\| \xi_{u}+\chi\right\| _{\infty}\leq\max\chi-\min\chi.
\end{equation*}
\end{remark}

\begin{lemma}
[Properties of $\eta_{u}$]For every $u\in \Lambda$,
\begin{equation}                                                                              \label{eta1}
\left\| \eta_{u} \right\| _{2}\geq\frac{\left| \kappa\right|
\left| \Omega\right| }{q^2\left\| u\right\| _{4}^{2}},
\end{equation}
\begin{equation}                                                                              \label{eta2}
\kappa\eta_{u}\geq0
\end{equation}
a.e. in $\Omega$ and
\begin{equation}                                                                             \label{spezz}
\left\| \nabla\eta_{u}\right\| _{2}\leq c_1 \left| \bar\eta_{u}
\right| \left\| u\right\| _{4}^{2}.
\end{equation}
\end{lemma}

\begin{proof}
Let $u\in \Lambda$ be fixed.
If $\kappa=0$, the lemma is trivial. So we suppose $\kappa\neq 0$.

By integrating the equation in (\ref{k}) on $\Omega$ we get
\[
q^2\int_\Omega \eta_u u^2\,dx=\kappa\left|\Omega \right|,
\]
from which we deduce (\ref{eta1}).

Moreover, since the unique solution $\eta_u$ of (\ref{k}) is the minimizer of
\begin{equation*}
f^*\left( \eta \right)=\frac{1}{2}\int_\Omega \left|\nabla \eta\right|^2\,dx+
\frac{q^2}{2}\int_\Omega \eta^2u^2\,dx-\kappa\left|\Omega\right|\bar\eta,
\end{equation*}
with analogous arguments to those used in the proof of (\ref{m}), we have that:
\begin{itemize}
\item if $\kappa<0$, then $\eta_u\leq 0$ a.e. in $\Omega$;
\item if $\kappa>0$, then $\eta_u\geq 0$ a.e. in $\Omega$.
\end{itemize}
Finally, multiplying the equation in (\ref{k}) by $\eta_u-\bar\eta_u$ and integrating, we get
\[
-\left\|\nabla\eta_u\right\|_2^2-q^2\int_\Omega\eta_u\left(\eta_u-\bar\eta_u \right)u^2\,dx=0
\]
from which
\[
\left\|\nabla\eta_u\right\|_2^2+q^2\int_{\Omega}\left(\eta_u-\bar\eta_u \right)^2 u^2\,dx=
-\bar\eta_u \int_{\Omega}\left(\eta_u-\bar\eta_u \right) u^2\,dx.
\]
Then, by the H\"older and Poincar\'e-Wirtinger inequalities, we obtain
\[
\left\|\nabla\eta_u\right\|_2^2\leq
\left|\bar\eta_u \right|\left\|\eta_u-\bar\eta_u \right\|_2\left\|u\right\|_4^2\leq
c_1\left|\bar\eta_u \right|\left\|\nabla\eta_u\right\|_2\left\|u\right\|_4^2
\]
which implies (\ref{spezz}).
\end{proof}

Finally we have the following relation between $\xi_{u}$ and $\eta_{u}$.

\begin{lemma}
For every $u\in \Lambda$,
\begin{equation}                                                          \label{misto}
q^2\int_{\Omega}\chi\eta_{u}u^{2}\,dx=-\kappa\left|\Omega\right|\bar\xi_{u}.
\end{equation}

\end{lemma}

\begin{proof}
Fixed $u\in \Lambda$, multiplying the equation of (\ref{onlychi}) by
$\eta_u$ and integrating on $\Omega$, we get
\[
-\int_\Omega \nabla\xi_u\nabla\eta_u\,dx-q^2\int_\Omega \xi_u\eta_u u^2\,dx=
q^2\int_\Omega \chi\eta_u u^2\,dx.
\]
Multiplying the equation of (\ref{k}) by $\xi_u$ and integrating on $\Omega$, we obtain
\[
-\int_\Omega \nabla\xi_u\nabla\eta_u\,dx-q^2\int_\Omega \xi_u\eta_u u^2\,dx=-
\kappa\left|\Omega\right|\bar\xi_{u}.
\]
The claim immediately follows.
\end{proof}

\section{Proof of Theorem \ref{thlin}}
Taking into account Remark \ref{gil}, in this section we assume that $\left\|h\right\|_{H^{1/2}\left(\partial\Omega\right)}$ is sufficiently small in order to get
\[
 \left\|\chi\right\|_\infty\leq m/q,
\]
hence
\begin{equation} \label{small}
 m^2-q^2\chi^2\geq 0.
\end{equation}

\subsection{Existence of nontrivial solutions}

In this subsection we assume that
$\int_{\partial\Omega}h~d\sigma\neq0$.

We give the explicit expression of the functional
$J(u)=F(u,\varphi_u)$. If $u\in \Lambda$, multiplying (\ref{120}) by
$\varphi_{u}$ and integrating on $\Omega$, we have
\begin{equation*}
-\left\| \nabla\varphi_{u}\right\| _{2}^{2} = q^2
\int_{\Omega}\varphi_{u}\left( \varphi_{u}+\chi\right) u^{2}\,dx-\kappa\left|\Omega\right|\bar\varphi_{u}.
\end{equation*}
Then, taking into account (\ref{aggiunta1}) and (\ref{misto}), we obtain
\begin{equation}                                                                       \label{i}
J\left( u\right) =\frac{1}{2}\left\| \nabla u\right\| _{2}^{2}+\frac{1}{2}\int_{\Omega
}\left(m^{2}-q^2\chi^{2}\right) u^{2}\,dx -\frac{q^2}{2}\int_{\Omega}\xi_{u} \chi u^{2}
\,dx+\kappa\left|\Omega\right|\bar\xi_{u}+\frac{\kappa\left|\Omega\right|}
{2}\bar\eta_{u}.
\end{equation}
Moreover, for every $v\in H^1_0\left(\Omega\right)$,
\begin{equation}                                                                          \label{jprimolin}
 \left\langle J'\left(u\right),v\right\rangle =\left\langle F'(u,\varphi_u),v\right\rangle=\int_\Omega\nabla u\nabla v\,dx+\int_\Omega \left[m^2-q^2\left(\varphi_u+\chi\right)^2\right]uv\,dx.
\end{equation}

\begin{proposition}                                                                       \label{propJ}
The functional $J$ has the following properties:
\begin{enumerate}
\item[\emph{(a)}] $\displaystyle{\lim_{u\rightarrow0}J(u)}=+\infty$,
\item[\emph{(b)}] $J$ is coercive,
\item[\emph{(c)}] $J$ is bounded from below.
\end{enumerate}
\end{proposition}
\begin{proof}
Assume $u\rightarrow 0$. Since the first four terms in (\ref{i}) are bounded from below, we study the last term. By (\ref{eta2}),
\begin{equation}\label{29'}
\frac{\kappa\left|\Omega\right|}{2}\,\bar\eta_{u}\geq 0.
\end{equation}
We claim that $\left|\bar\eta_u\right|\rightarrow +\infty$.

Arguing by contradiction, assume that there exists a sequence
$u_n\rightarrow 0$ such that $\{\bar\eta_n\}$ is bounded (where we
mean $\eta_n=\eta_{u_n}$). Hence, by (\ref{spezz}), we have
$\left\|\nabla\eta_n\right\|_2\rightarrow 0 $. Then, using the
Poincar\'e-Wirtinger inequality, we deduce that $\{\eta_n\}$ is
bounded. On the other hand (\ref{eta1}) yields
\begin{equation*}
 \lim_n\left\|\eta_n\right\|_2=+\infty,
\end{equation*}
so we get a contradiction and (a) is proved.

By (\ref{intmeno}), (\ref{small}) and (\ref{29'}), we obtain
\begin{equation*}
 J\left(u\right)\geq \frac{1}{2}\left\|\nabla u\right\|_2^2+\kappa\left|\Omega\right|\bar\xi_{u}.
\end{equation*}
Then, by (\ref{stimaphi}), we deduce (b) and (c).
\end{proof}

\begin{proposition}
The functional J satisfies the Palais-Smale condition on $\Lambda$, i.e. every sequence $\left\{ u_{n}\right\} \subset \Lambda$ such that
$\left\{J(u_n)\right\}$ is bounded and $J^{\prime}\left( u_{n}\right)\rightarrow 0$, admits a converging subsequence in $\Lambda$.
\end{proposition}

\begin{proof}
Let $\left\{u_n\right\}\subset \Lambda$
be a Palais-Smale sequence, i.e.
\begin{equation}                                                                    \label{ips1}
 \left\{J\left(u_n\right)\right\} \textrm{ bounded}
\end{equation}
and
\begin{equation*}											
 J'\left(u_n\right)\rightarrow 0.
\end{equation*}
>From (\ref{ips1}) and (b) of Proposition \ref{propJ} we deduce that $ \left\{u_n\right\}$ is bounded, hence it converges weakly to $u\in H_0^1(\Omega)$.
It remains to prove that the convergence is strong and that $u\neq0$.
As before, for the sake of simplicity, we set $\varphi_n=\varphi_{u_n}$, $\xi_n=\xi_{u_n}$ and
$\eta_n=\eta_{u_n}$.

By (\ref{jprimolin}) and (\ref{aggiunta1}), we have
\begin{equation}                                                                      \label{iprimo}
\Delta u_n =m^2 u_n  -q^2\left(\xi_n + \eta_n + \chi\right)^2 u_n  -J^{\prime}\left(u_n\right).
\end{equation}
So it is sufficient to prove that the right
hand side of (\ref{iprimo}) is bounded in $H^{-1}\left(\Omega\right)$. 
Since $u_n \rightharpoonup u$ and $J'\left(u_n\right)\rightarrow 0$, we have only to study 
$\{ \left(\xi_n + \eta_n + \chi\right)^2 u_n\}$.
>From (\ref{ips1}) we deduce that $\left\{\kappa\left|\Omega\right|\bar\eta_n/2\right\}$ is bounded, the same being true for the first four terms in $J(u_n)$.
Then, using (\ref{spezz}), we conclude that $\left\{\eta_n\right\}$ is bounded, as well as $\left\{\xi_n\right\}$ by (\ref{stimaphi}). The claim easily follows.

Finally (a) of Proposition \ref{propJ} and (\ref{ips1}) show that $u$ cannot be zero. The proof is thereby complete.
\end{proof}

Using again (a) of Proposition (\ref{propJ}), we can see that the sublevels of $J$ are complete. Then, by a standard tool in critical point theory (Deformation Lemma, see \emph{e.g.} \cite{rab}), we conclude that the minimum of $J$ is achieved.

\subsection{The \emph{only if} part}

In this subsection we show that if $\int_{\partial\Omega}h\,d\sigma=0$, then problem (\ref{probchilin}) has only trivial solutions.

Let $\left(  u,\varphi\right)$
be a solution of (\ref{probchilin}) with $\kappa=0$. By the first equation we have
\begin{equation}
\left\Vert \nabla u\right\Vert _{2}^{2}-q^2\int_{\Omega}\left(  \varphi
+\chi\right)  ^{2}u^{2}\,dx+m^{2}\left\Vert u\right\Vert _{2}^{2}=0.
\label{banal10}
\end{equation}

By the second equation we have

\begin{equation}                                                                 \label{banal17}
-\left\| \nabla\varphi\right\|_2^2-q^2\int_{\Omega
}u^{2}\varphi^{2}dx  =q^2\int_{\Omega}\chi\varphi u^{2}dx.
\end{equation}
Then, substituting $\int_{\Omega}\chi\varphi u^{2}dx$ in (\ref{banal10}), we obtain
\begin{equation*}                                                                 
\left\Vert \nabla u\right\Vert _{2}^{2}+q^2\int_\Omega u^2 \varphi^2\,dx+\int_{\Omega}\left(  m^{2}
-q^2\chi^{2}\right)  u^{2}dx+2\left\Vert \nabla\varphi
\right\Vert _{2}^{2}=0.
\end{equation*}

Therefore, taking into account (\ref{small}), we deduce $u=0$.

\section{Proof of Theorem \ref{main}}
In this section we assume $\kappa=0$, so we have
\begin{equation}                                                                                                                                    \label{eta0}
\varphi_u=\xi_u.
\end{equation}

Since $\varphi_u$ satisfies (\ref{banal17}), substituting in (\ref{defj}), we find, for every $u\neq0$,
\begin{equation*}                                                                        
J_{g}\left( u\right) =\frac{1}{2}\left\| \nabla u\right\|
^{2}_2+ \frac{m^2}{2}\int_{\Omega} u^{2}\,dx-\frac{q^2}{2}\int_{\Omega}
\chi\left( \varphi_{u}+\chi\right) u^{2}\,dx-\int_{\Omega} G\left( x,u\right)
\,dx
\end{equation*}
and
\begin{equation}                                                                       \label{jprimo}
 \left\langle J_{g}'\left(u\right),v\right\rangle =\int_\Omega\nabla u\nabla v\,dx+\int_\Omega \left[m^2-q^2\left(\varphi_u+\chi\right)^2\right]uv\,dx-\int_\Omega g\left(x,u\right)v\,dx
\end{equation}
for $v\in H^1_0\left(\Omega\right)$.

About the nonlinear term, we recall that $\mathbf{\left( g_{1}\right) }-\mathbf{\left(
g_{3}\right) }$ imply that:
\begin{enumerate}
 \item[$\left(\textbf{G}_1\right)$] for every $\varepsilon>0$ there exists $A\geq0$
such that for every $t\in\mathbf{R}$
\[
\left|  G\left( x,t\right)  \right|  \leq\frac{\varepsilon}{2}
t^{2}+A\left| t\right| ^{p};
\]
\item[$\left(\textbf{G}_2\right)$] there exist two constants $b_{1},b_{2}>0$ such that
for every $t\in\mathbf{R}$
\[
G\left( x,t\right)  \geq b_{1}\left| t\right| ^{s}-b_{2}.
\]
\end{enumerate}

This time the functional has not a singularity in $0$, but it can be extended according
to the following proposition.

\begin{proposition}
If we set $J_{g}\left( 0\right) =0$, then the functional $J_g$ is $C^{1}$ on
$H_{0}^{1}\left( \Omega\right) $ with $J_g^{\prime}\left( 0\right) =0$.
\end{proposition}

\begin{proof}
>From (\ref{stimaphi}) and (\ref{eta0}) we deduce
\begin{equation}                                                                  \label{limint}
\left|\int_{\Omega }\chi\left(\varphi_u+\chi\right)u^2\,dx\right|
\leq 2\left\|\chi\right\|_\infty ^2\left\| u\right\|_2^2.
\end{equation}
Then it is easy to see that
\[
\lim_{u\rightarrow 0}J_{g}\left(u\right)=0,
\]
hence $J_g$ is continuous on $H_{0}^{1}\left(\Omega\right)$.

Using again (\ref{limint}) and $\left(\textbf{G}_1\right)$, we obtain
\[
\lim_{u\rightarrow 0}\frac{J_{g}\left(u\right)}{\left\|\nabla u\right\|_2}=0,
\]
which, joint with $J_{g}\left(0\right)=0$, implies that $J_{g}$ is differentiable
in $0$ and $J_{g}'\left(0\right)=0$.

Finally, we have that $J_{g}'$ is continuous in $0$. Indeed, from (\ref{jprimo}) we get
\[
\left|\left<J_{g}'\left(u\right),v\right>\right|\leq\left\|\nabla u
\right\|_2\left\|\nabla v\right\|_2
+\left(4q^2\left\|\chi\right\|_\infty^2+m^2\right)\left\|u\right\|_2\left\|v\right\|_2
+\int_{\Omega }\left|g\left(x,u\right)v\right|\,dx.
\]
Then, using the hypotheses on $g$,
\[
\lim_{u\rightarrow 0}\left\|J_{g}'\left(u\right)\right\|=\lim_{u\rightarrow 0}\sup_{\left\|\nabla v\right\|_2=1}\left|\left<J_{g}'\left(u\right),v\right>\right|=0.
\]
\end{proof}

\begin{proposition}
The functional $J_g$ satisfies the Palais-Smale condition on $H^{1}_{0}\left(
\Omega\right) $.
\end{proposition}

\begin{proof}
Let $\left\{u_{n}\right\}\subset H^1_0\left(\Omega\right)$ such that
\begin{eqnarray}
&& \left|J_{g}\left(u_n\right)\right|\le c                                                               \label{PS1}
\\
&& J_{g}^{\prime}\left(u_{n}\right)\rightarrow 0 .                                                          \label{PS2}
\end{eqnarray}

As before, we set $\varphi_n=\varphi_{u_n}$ and we use $c_i$ to denote suitable positive constants.
By (\ref{PS1})
\begin{eqnarray}
\frac{1}{2}\left\|\nabla u_n\right\|_2^2 &\le& c +\int_\Omega G\left(x,u_n\right)\,dx
+\frac{q^2}{2}\int_\Omega \chi\left(\varphi_n+\chi\right)u_n^2\,dx +\frac{m^2}{2}\|u_n\|^{2}_{2}
\nonumber \\
& \le & c_1+ \frac{1}{s}\int_{\left\{x\in\Omega :\left|u_n\left(x\right)\right|\ge r\right\}}
g\left(x,u_n\right)u_n\,dx
+c_2\left\|u_n\right\|_2^2 \nonumber  \\
& \le &c_3+ \frac{1}{s}\int_\Omega g\left(x,u_n\right)u_n\,dx
+c_2\left\|u_n\right\|_2^2 .                                             \label{1relaz}
\end{eqnarray}
On the other hand, by (\ref{jprimo}) and (\ref{PS2}),
\begin{align*}
\left|\left\langle J_{g}^{\prime}\left(u_n\right),u_n\right\rangle\right|=&
\left|\left\| \nabla u_n\right\|_2^2+m^2\left\|u_n \right\|^2_2
-q^2\int_\Omega\left(\varphi_n+\chi\right)^2u_n^2\,dx
-\int_\Omega g\left(x,u_n\right)u_n\,dx \right|\\
\leq &  \,c_4 \left\|\nabla u_n\right\|_2
\end{align*}
and so
\begin{eqnarray}
\int_\Omega g\left(x,u_n\right)u_n\,dx
& \le & c_4\left\|\nabla u_n\right\|_2+\left\|\nabla u_n\right\|_2^2+m^2\left\|u_n\right\|_2^2
-q^2\int_\Omega\left(\varphi_n+\chi\right)^2u_n^2\,dx  \nonumber \\
& \le & c_4\left\|\nabla u_n\right\|_2+\left\|\nabla u_n\right\|_2^2+m^2\left\|u_n\right\|_2^2.                                                      \label{2relaz}
\end{eqnarray}
Hence, substituting (\ref{2relaz}) in (\ref{1relaz}) we easily find
\begin{equation}
\frac{s-2}{2s}\left\|\nabla u_n\right\|_2^2\le c_3
+\frac{c_4}{s}\left\|\nabla u_n\right\|_2+\|\chi\|^2_\infty\left\|u_n\right\|_2^2
+m^{2}\frac{s+2}{2s}\|u_n\|_2^2.                   \label{priori}
\end{equation}
Now we claim that $\left\{u_n\right\}$ is bounded in $H^1_0\left(\Omega\right)$.
Otherwise by (\ref{priori})
$$
\left\| u_n\right\|_2^2\ge c_5\left\|\nabla u_n\right\|_2^2
-c_6\left\|\nabla u_n\right\|_2-c_7
$$
and, for $n$ sufficiently large, we have
\[
 \left\|u_n\right\|_2^2\ge c_8\left\|\nabla u_n\right\|_2^2\rightarrow +\infty.
\]
So, using  $\left(\textbf{G}_2\right)$ and (\ref{limint}), we deduce
\begin{align*}
J_{g}\left(u_n\right)& =  \frac{1}{2}\left\|\nabla u_n\right\|_2^2
+\frac{m^2}{2}\left\| u_n \right\|_2^2
-\frac{q^2}{2}\int_\Omega\chi\left(\varphi_n+\chi\right)u_n^2\,dx
-\int_\Omega G\left(x,u_n\right)\,dx  \nonumber \\
& \le  \frac{1}{2}\left\|\nabla u_n\right\|_2^2+
c_9\left\|u_n\right\|^2_2
-b_1\left\|u_n\right\|_s^s+b_2\left|\Omega\right|   \nonumber \\
& \le c_{10}\left\|u_n\right\|^2_2-c_{11}\left\|u_n\right\|^s_2+
b_2\left|\Omega\right|\rightarrow -\infty,
\end{align*}
which contradicts (\ref{PS1}).
So $\left\{u_n\right\}$ is bounded and, up to a subsequence,
\begin{equation*}
u_n\rightharpoonup u \mbox{ in } H^1_0\left(\Omega\right) .
\end{equation*}

We have to prove that the convergence is strong.
We know that
\begin{equation}
\Delta u_n = m^2u_n  -q^2\left(\varphi_n +\chi\right)^2 u_n -
g\left(x,u_n\right) -J_{g}^{\prime}\left(u_n\right).                                                      \label{J'}
\end{equation}
The sequences $\left\{J^{\prime}\left(u_n\right)\right\}$, $\left\{ u_n\right\} $ and
$\left\{g\left(x,u_n\right)\right\}$ are bounded.
Finally, by Corollary \ref{estimates}, $\{\varphi_n + \chi \}$ is bounded
in $L^{\infty}\left(\Omega\right)$, then
$\{ (\varphi_n+\chi)^2 u_n\}$
is bounded in $L^2\left(\Omega\right)$.
Therefore the right hand side of (\ref{J'}) is a bounded sequence in
$H^{-1}\left(\Omega\right)$. By standard arguments the proof is complete.
\end{proof}

Finally we notice that, by $\left(\textbf{G}_2\right)$,
\begin{align*}
J_{g}\left( u\right)  &  \leq \frac{1}{2}\left\| \nabla u\right\| _{2}^{2}+
\left(q^2\left\|\chi\right\|^2_\infty+\frac{m^2}{2}\right)\left\| u\right\| _{2}^{2} -\int_{\Omega}G(x,u)\,dx\\     \nonumber
&   \le \frac{1}{2}\left\| \nabla u\right\| _{2}^{2}+ \left(q^2\left\|\chi\right\|^2_\infty+\frac{m^2}{2}\right) \left\| u\right\|
_{2}^{2} -b_{1}\left\| u\right\| _{s}^{s}+b_{2}\left| \Omega\right|.
\end{align*}
Hence, if $V$ is a finite dimensional subspace of $H^{1}_{0}\left( \Omega\right) $, then

\begin{equation} \label{dim fin}
\lim_{\substack{ u\in V \\ \left\| \nabla u\right\|_2 \rightarrow +\infty }}
J_g\left(u\right)=-\infty.
\end{equation}

\subsection{Proof of (a)}
Let $\left\{ \lambda_{j}\right\} $ denote the sequence of the eigenvalues
of $-\Delta$ with Dirichlet boundary conditions. Taking into account Remark \ref{gil}, assume that
\[
q^2\left\| \chi\right\| ^{2}_{\infty}<\lambda_{1}+m^{2}.
\]

>From (\ref{eta0}), (\ref{intmeno}) and $\left(\textbf{G}_1\right)$ we deduce
\begin{align*}
J_{g}\left( u\right)  &  \geq
\frac{1}{2}\left[ \left\| \nabla u\right\| _{2}^{2}
+\left(m^{2}-q^2\left\| \chi\right\|^2_{\infty}\right)
\left\| u\right\| _{2}^{2}\right] -\frac{\varepsilon}{2}\left\| u\right\|
_{2}^{2} -A\left\| u\right\| _{p}^{p}\\
&  \ge\frac{\lambda_{1}+m^{2}-q^2\left\| \chi\right\| ^{2}_{\infty}-\varepsilon} {2\lambda_{1}%
}\left\| \nabla u\right\| _{2}^{2}-A^{\prime}\left\| \nabla u\right\| _{2}^{p},%
\end{align*}
with $A,A'>0$ depending on $\varepsilon>0$. Choosing $\varepsilon$
sufficiently small, we deduce
\[
J_{g}\left( u\right)  \ge c\left\| \nabla u\right\| _{2}^{2}-A^{\prime}\left\|
\nabla u\right\| _{2}^{p}
\]
with $c>0$. Hence $J_g$ has a strict local minimum in $0$.

Taking into account (\ref{dim fin}), the classical Mountain Pass
Theorem of Am\-bro\-set\-ti-Rabinowitz applies (see \emph{e.g.}
\cite{rab}) and we deduce the existence of a nontrivial solution.

\subsection{Proof of (b)}

Since $g$ is odd, the functional $J_g$ is even and we use the
$\mathbf{Z}_{2}$-Mountain Pass Theorem as stated in \cite{rab}.

\begin{theorem}
\label{Z2MP} Let $E$ be an infinite dimensional Banach space and let $I\in
C^{1}\left(  E,\mathbf{R}\right)  $ be even, satisfy the Palais-Smale
condition and $I\left(  0\right)  =0.$ If $E=V\oplus X,$ where $V$ is finite
dimensional and $J$ satisfies

\begin{enumerate}
\item there are constants $\rho,\alpha>0$ such that $\left.  I\right\vert
_{\partial B_{\rho}\cap X}\geq\alpha,$ and

\item for each finite dimensional subspace $\tilde{E}\subset E,$ there is an
$R=R(\tilde{E})$ such that $I\leq0$ on $E\setminus B_{R(\tilde{E})},$
\end{enumerate}
then $I$ possesses an unbounded sequence of critical values.
\end{theorem}

Taking into account (\ref{dim fin}), in order to apply Theorem \ref{Z2MP}, we have to prove the geometrical property stated in \emph{(1)}.

We distinguish two cases:
\begin{enumerate}
 \item [(a)] If $q^2\left\| \chi\right\| ^{2}_{\infty}-m^{2}<\lambda_{1}$ then,
using the same estimates given in the previous subsection,
Theorem \ref{Z2MP} applies with $V=\left\{ 0\right\} $.

\item [(b)] If $\lambda_{1}\le q^2\left\| \chi\right\| ^{2}_{\infty}-m^{2}$, we set
\[
k=\min\left\{ j\in\mathbf{N}:\;q^2\left\| \chi\right\| ^{2}_{\infty}-m^{2}<\lambda
_{j}\right\} ,
\]
and we consider
\[
V=\bigoplus_{j=1}^{k-1} M_{j},\;\;\;\;\;X=V^{\perp}=\overline{\bigoplus
_{j=k}^{+\infty}M_{j}}.
\]
where $M_{j}$ is the finite dimensional eigenspa\-ce corresponding to $\lambda_j$.

Since
\[
\lambda_{k}=\min\left\{ \frac{\left\| \nabla v\right\| _{2}^{2}}{\left\|
v\right\| ^{2}_{2}}:\; v\in X,\; v\neq0\right\} ,
\]
for every $u\in X$ we have
\[
J_{g}\left( u\right) \geq\frac{\lambda_{k}+m^{2}-q^2\left\| \chi\right\| ^{2}_{\infty}}
{2\lambda_{k}}\left\| \nabla u\right\| _{2}^{2}-\int_{\Omega}G(x,u)dx.
\]
Similar estimates to those used in the previous case show that $J$ is strictly
positive on a sphere in $X$.
\end{enumerate}

In both cases we get the existence of infinitely many critical points $\{{u_i}\}$ such that
$$
J_{g}(u_i)\rightarrow +\infty.
$$
Remark \ref{unifbound} gives the uniform estimate on $\{\varphi_{u_i}\}$.
Finally we notice that,
by $\left(\textbf{G}_1\right)$,
\begin{align*}
J_{g}\left(u_i\right) & =  \frac{1}{2}\left\|\nabla u_i\right\|_2^2
+\frac{m^2}{2}\left\| u_i \right\|_2^2
-\frac{q^2}{2}\int_\Omega\chi\left(\varphi_i+\chi\right)u_i^2\,dx
-\int_\Omega G\left(x,u_i\right)\,dx\\
& \le  c_1\|\nabla u_i \|_2^2+c_2\left\| \nabla u_i\right\|_2^p.
\end{align*}
Hence $\left\|\nabla u_i\right\|_2\to + \infty$ and this completes
the proof.

\end{document}